\documentclass[pdflatex,sn-mathphys-num]{sn-jnl}

\usepackage{graphicx}
\usepackage{multirow}
\usepackage{amsmath,amssymb,amsfonts}
\usepackage{amsthm}
\usepackage{mathrsfs}
\usepackage[title]{appendix}
\usepackage{xcolor}
\usepackage{textcomp}
\usepackage{manyfoot}
\usepackage{booktabs}
\usepackage{algorithm}
\usepackage{algorithmicx}
\usepackage{algpseudocode}
\usepackage{listings}
\usepackage{tikz-cd}
\usepackage[shortlabels]{enumitem}

\theoremstyle{thmstyleone}
\newtheorem{theorem}{Theorem}
\newtheorem{proposition}[theorem]{Proposition}
\newtheorem{lemma}[theorem]{Lemma}
\theoremstyle{thmstyletwo}
\newtheorem{example}{Example}
\newtheorem{remark}{Remark}
\theoremstyle{thmstylethree}
\newtheorem{definition}{Definition}
\begin{document}

\title[On the Iterates of Endomorphisms of Projective Varieties]{On the Iterates of Endomorphisms of Projective Varieties}

\author{\fnm{Sami} \sur{al-Asaad}}

\abstract{We initiate a study of certain notions related to the algebraic dynamics of endomorphisms of projective varieties. Our main results include: (i) a characterization of projective varieties that occur as iterated images of endomorphisms of normal projective varieties; (ii) exploring stability behavior in the Stein factorizations of the iterates; and (iii) a structure result concerning endomorphisms whose iterates lie in a finite union of connected components of the endomorphism scheme, thereby completing a result of Brion.
}


\maketitle

\section{Introduction of Main Results}\label{sec1}
We adopt the following conventions throughout this article: all schemes and morphisms of schemes are assumed to be over a fixed ground field $k$ of arbitrary characteristic. All schemes are assumed separated. A \emph{variety} is a geometrically integral scheme of finite type. A \emph{curve} is a variety of dimension one. We use \cite{hartalggeom} as a general reference for algebraic geometry.

A \emph{semigroup} is a set endowed with an associative law of composition. An element $e$ of a semigroup is called idempotent if $e^2=e$. A \emph{monoid} is a semigroup having a neutral element (whose uniqueness is then evident). A subset $G$ of a monoid $M$ is called a \emph{subgroup} of $M$ if $G$ is a group under the law of $M$ (hence the neutral element of $G$ may differ from that of $M$).

A group scheme is a group object in the category of schemes. Monoid schemes and semigroup schemes are defined similarly. We denote the neutral component of a group scheme $G$ by $G^0$. A (locally) algebraic group is a group scheme (locally) of finite type. (Locally) algebraic monoids and semigroups are defined in a similar way. 

Let $X$ be a projective variety, and let $f:X\to X$ be an endomorphism of $X$. The pair $(X,f)$ is called an \emph{(algebraic) dynamical system}. In this section, we introduce some stability phenomena that the iterates $f^n$ exhibit for $n$ large enough. While objects such as the iterated image $\bigcap_{n} f^n(X)$, the Stein factorizations of the iterates $f^n$, and the schematic closure of the cyclic semigroup $\langle f \rangle$ in the endomorphism scheme, appear naturally in the context of algebraic dynamics, they do not seem to have been studied systematically, to the best of our knowledge.

The article is organized as follows. This section contains the statements of our main results, while Sections \ref{proofsitim} through \ref{AE} are devoted to their proofs, as well as further comments and future perspectives.

\subsection{Iterated Images}
As a first example, note that the following decreasing chain of closed subvarieties of $X$:
\[
    X \supseteq f(X) \supseteq f^2(X) \supseteq f^3(X) \supseteq  \cdots
\]
must stabilize. We call the closed subvariety of $X$ at which this chain stabilizes the \emph{iterated image} of $f$ or of $(X,f)$. We collect some basic properties in the following

\begin{lemma}\label{fny}
Let $Y$ be the iterated image of the dynamical system $(X,f)$. Then $Y = f^m(Y)$ for all $m$. It follows that $f$ induces a surjective and finite endomorphism ${f^m}^{|Y} \circ i : Y \to X \to Y$, where ${f^m}^{|Y}$ and $i$ are the restriction of $f^m$ to its image and the inclusion, respectively. Denoting ${f^m}^{|Y} \circ i$ by $f^m_Y$, we note that $f^m_Y = (f_Y)^m$ and $f^m_Y \circ f^n_Y = f^{m+n}_Y$ for all $m,n$.
\end{lemma}

Indeed, for $n \gg 1$, we have $f^n(X) = Y$, yielding $Y = f^{n+m}(X)= f^m(Y)$ for all $m$.

It follows readily that every projective variety $Y$ arises as the iterated image of some dynamical system $(X,f)$: simply take $X = Y$ and $f=\mathrm{id}_X$. We are thus naturally led to impose additional conditions; we shall consider projective varieties $Y$ satisfying the following property:
\begin{equation}\label{propertystar}
Y \text{ is the iterated image of some dynamical system } (X,f), \text{ with } X  \text{ normal.} 
\tag{$\mathcal{I}$}
\end{equation}

\begin{theorem}\label{itimage}
A projective variety $Y$ satisfies Property \eqref{propertystar} if and only if there exists a finite morphism $Y\to \tilde{Y}$ to the normalization $\tilde{Y}$.
\end{theorem}

Property \eqref{propertystar} has a more explicit characterization in the case of singular curves over an algebraically closed field of characteristic zero:

\begin{proposition}\label{pinchprop}
Suppose $k$ is algebraically closed and of characteristic zero. Let $Y$ be a projective curve, singular of geometric genus $g$, and let $\nu: \tilde{Y} \to Y$ be the normalization morphism.
\begin{enumerate}[(1)]
    \item If there exists a finite morphism $h: Y \to \tilde{Y}$, then either $g =0$ or $g=1$.
    \item In the case $g=0$, there exists a finite morphism $h: Y \to \tilde{Y}$.
    \item In the case $g=1$:
        \begin{enumerate}
            \item If $Y$ admits a finite morphism $h: Y \to \tilde{Y}$, then there exists a $k$-rational point $y \in \tilde{Y}(k)$ such that $h \circ \nu$ is an isogeny of the elliptic curve $(\tilde{Y}, y)$, and every fiber of $\nu$ is contained in a translate of $\ker(h \circ \nu)$.
            \item Let $F_1, \ldots, F_m \subseteq \tilde{Y}(k)$ be mutually disjoint (set-theoretic) fibers of $\nu$ (all of which must be finite since $\nu$ is finite). Then $Y$ admits a finite morphism $h: Y \to \tilde{Y}$ if and only if there exists a finite subgroup $F \subseteq \tilde{Y}(k)$ such that each $F_i$ is contained in a translate of $F$.
        \end{enumerate}
\end{enumerate}
\end{proposition}

Proposition \ref{pinchprop} does not hold in positive characteristic as Example \ref{pinchingcounterex} will demonstrate. This is done in Section \ref{proofsitim}.

In dimension two, normal projective surfaces admitting a non-invertible surjective endomorphism have been classified in \cite{nobnak}. While it remains unclear to us whether this classification can yield a result analogous to Proposition \ref{pinchprop}, investigating it would be a natural first step.

\subsection{Dynamical Stein Factorizations}
Recall from \cite[Corollary III.11.5]{hartalggeom} that every morphism $f:X \to Y$ of projective varieties factors uniquely as 
\begin{equation}\label{stein}
    \begin{tikzcd}[arrows={-stealth}]
    f:X \ar[r,"g"] & Z \ar[r,"h"] & Y
    \end{tikzcd}
\end{equation}
where $g$ is a \emph{contraction}, i.e. $g_*\mathscr{O}_X = \mathscr{O}_Z$, and $h$ is finite. We call (\ref{stein}) the \emph{Stein factorization} of $f$, and it follows that $g$ is projective with geometrically connected fibers, $Z = \mathrm{Spec} f_*\mathscr{O}_X$, and $Z$ is the normalization of $Y$ in $X$. It turns out that examining the Stein factorizations of the iterates in a given dynamical system reveals further stability phenomena. 

\begin{definition}\label{dynstein}
Let $(X,f)$ be a dynamical system. We call the Stein factorization of the iterate $f^n$ the $n$th \emph{Stein factorization} of $f$ or of $(X,f)$.
\begin{equation}\label{nstein}
    \begin{tikzcd}[arrows={-stealth}]
    f^n:X \ar[r,"g_n"] & Z_n \ar[r,"h_n"] & X
    \end{tikzcd}
\end{equation}
Moreover, we call $Z_n$ the $n$th \emph{Stein factor} of $f$ or of $(X,f)$.
\end{definition}

It may be observed from (\ref{nstein}) and the surjectivity of $g_n$ that $h_n(Z_n) = h_n(g_n(X)) = f^n(X)$. Hence, denoting by $h_n'$ the restriction of $h_n$ to its image, we can ``sharpen'' (\ref{nstein}) as follows: 
\begin{equation}\label{nstein2}
    \begin{tikzcd}[arrows={-stealth}]
    X \ar[rr, "f^n"] \ar[dd, "g_n"'] && X \\
    \\
    Z_n \ar[rr, "h'_n"] && f^n(X) \ar[uu, hook]
    \end{tikzcd}
\end{equation}
with $h'_n$ finite and surjective. In particular, $\dim Z_n = \dim f^n(X)$, and the sequence of images $(h_n(Z_n))=(h'_n(Z_n))$ stabilizes (at the same time the chain $(f^n(X))$ does). It also follows that if $Y$ is the iterated image of $f$, then $\dim Z_n \geq \dim Y$ for all $n$. In Propositions \ref{stability} and \ref{steinfactoriteratedimage}, we describe the behavior of the sequences $(g_n)$, $(Z_n)$, and $(\deg(h'_n))$.

\begin{proposition}\label{stability}
Let $(X,f)$ be a dynamical system, and consider its sequence of Stein factorizations \eqref{nstein2}. Then the sequences $(g_n)$ and $(Z_n)$ of contractions and Stein factors, respectively, stabilize.
\end{proposition}

In this case, we call the Stein factor at which the chain $(Z_n)$ stabilizes the \emph{iterated Stein factor} of $f$ or of $(X,f)$.

\begin{proposition}\label{steinfactoriteratedimage}
Let $(X,f)$ be a dynamical system, with iterated image $Y$ and iterated Stein factor $Z$. 
\begin{enumerate}[(1)]
    \item For every $n \gg 1$, we have a commutative triangle
    \[
    \begin{tikzcd}[arrows={-stealth}]
       & Z \ar[dl, "h'_n"'] \ar[dr, "h'_{n+1}"] & \\
       Y \ar[rr, "f_Y"]  && Y
    \end{tikzcd}
    \]
    where $f_Y$ is the endomorphism induced by $f$, introduced in Lemma \ref{fny}. In particular, $\dim Z = \dim Y$, and $\deg(h'_{n+1}) = \deg(f_Y) \deg(h'_n)$, i.e. the numerical sequence $(\deg(h'_n))$ 
    becomes geometric starting from a sufficiently large $n$.
    \item For every $n \gg 1$, there exists an endomorphism $\varphi_n:Z \to Z$ lifting $f^n_Y$, making the square
    \[
    \begin{tikzcd}[arrows={-stealth}]
    Z \ar[rr, "\varphi_n"] \ar[dd, "h'_n"'] &&  Z \ar[dd, "h'_n"] \\
    &&\\
     Y \ar[rr, "f^n_Y"]   && Y
    \end{tikzcd}
    \]
    commute. In particular, $\deg(\varphi_n) = \deg(f_Y)^n$.
\end{enumerate}
\end{proposition}

\subsection{Bounded Endomorphisms}
Some generalities are in order. Given a projective variety $X$, the contravariant functor sending a scheme $S$ to the monoid of $S$-endomorphisms of $X \times_k S$, is representable by a locally algebraic monoid, denoted by $\mathrm{End}_X$, and called the \emph{endomorphism scheme} of $X$ (see \cite[Theorem 5.23]{fantechi2005fundamental} for this representability result under the more general assumption that $X$ is a projective scheme over an arbitrary noetherian ground scheme). We have in particular $\mathrm{End}_X(k) = \mathrm{End}(X)$. By \cite[I,\S 4,6.5]{demazure1970groupes}, there exist an \'etale scheme $\pi_0\mathrm{End}_X$, called the \emph{scheme of connected components} of $\mathrm{End}_X$, and a morphism $q_{\mathrm{End}_X}:\mathrm{End}_X \to \pi_0\mathrm{End}_X$, called the \emph{canonical projection}, satisfying the following universal property: for every \'etale scheme $E$ and every morphism $f: \mathrm{End}_X \to E$, there is a unique morphism $g:\pi_0\mathrm{End}_X \to E$ such that $f=g\circ q_{\mathrm{End}_X}$. Furthermore, $\pi_0\mathrm{End}_X$ inherits a monoid structure from $\mathrm{End}_X$ \cite[I,\S 4,6.10]{demazure1970groupes}, and $q_{\mathrm{End}_X}$ is a homomorphism of monoid schemes. Lastly, $q_{\mathrm{End}_X}$ is faithfully flat, and its fibers are the connected components of $\mathrm{End}_X$. It is not hard to show that $q_{\mathrm{End}_X}$ is of finite type.

Similarly, we have the \emph{automorphism scheme} $\mathrm{Aut}_X$; a locally algebraic group, as well as the \'etale group scheme $\pi_0\mathrm{Aut}_X$ of connected components of $\mathrm{Aut}_X$, and the canonical projection $q_{\mathrm{Aut}_X}: \mathrm{Aut}_X \to \pi_0 \mathrm{Aut}_X$, which is the restriction of $q_{\mathrm{End}_X}$ to $\mathrm{Aut}_X$ by \cite[Remark 15(i)]{brion2014algebraic}.

Finally, there exists a \emph{degree map} $\deg: \mathrm{End}_X \to \mathbb{N}$ that is locally constant (this is a consequence of the principle of conservation of number, see \cite[\S 10.2]{fulton2013intersection}), and such that $\mathrm{Aut}_X = \deg^{-1}(1)$.

On the abstract-algebraic level, recall that given a semigroup $S$ and $a \in S$, the set $\{ a^n \mid n \geq 1 \}$ is a subsemigroup of $S$, denoted by $\langle a \rangle$ and called the \emph{cyclic} (or \emph{monogenic}) subsemigroup of $S$ generated by $a$. The \emph{order} of $a$ is defined as the cardinality of $\langle a \rangle$. By \cite[\S 1.6]{clifford1961semigroups}, there are exactly two possibilities:
\begin{itemize}
    \item For every $i \neq j$ we have $a^i \neq a^j$. Then $\langle a \rangle$ is isomorphic to the additive semigroup $\mathbb{N}^*$. In particular, $\langle a \rangle$ contains no idempotents.
    \item There exists a smallest positive integer $s$ such that $a^s=a^r$ for some $r < s$. Then we can write
        \[
            \langle a \rangle =\{ a, a^2, \ldots , a^{r-1} \} \coprod K_a,
        \]
    where $a, a^2, \ldots , a^{r-1}$ are pairwise distinct, and $K_a := \{a^r, a^{r+1}, \ldots , a^{s-1}\}$ is a cyclic group. In particular, $\langle a \rangle$ contains a unique idempotent, namely the neutral element of the group $K_a$. We call the integers $r$ and $s-r$ the \emph{index} and \emph{period} of $a$, respectively. It follows that $\langle a \rangle$ is determined (up to isomorphism) by the index and period of $a$.
\end{itemize}

Given this setup, we formulate the following definition:

\begin{definition}\label{defsfgf}
Let $(X,f)$ be a dynamical system.
\begin{itemize}
    \item The schematic closure of the cyclic subsemigroup $\langle f \rangle$ in $\mathrm{End}_X$ is a reduced closed subsemigroup scheme of $\mathrm{End}_X$, that we denote by $S(f)$. If $f$ is an automorphism, we denote by $G(f)$ the schematic closure of the cyclic subgroup $\langle f \rangle$ in $\mathrm{Aut}_X$.
    \item We say that $f$ is \emph{bounded} (or \emph{algebraic}) if any of the following equivalent conditions is satisfied:
    \begin{enumerate}[(1)]
        \item $f$ belongs to an algebraic subsemigroup of $\mathrm{End}_X$.
        \item The iterates $f^n$ all belong to a finite union of connected components of $\mathrm{End}_X$.
        \item The subsemigroup scheme $S(f)$ of $\mathrm{End}_X$ is algebraic.
\end{enumerate}
    If $f$ is an automorphism, we say it is bounded (as an automorphism) if the subgroup scheme $G(f)$ of $\mathrm{Aut}_X$ is algebraic.
\end{itemize}
\end{definition}

The notion of bounded endomorphism was first introduced in \cite[Definition 1]{brion2014automorphisms}.

Clearly, both $S(f)$ and $G(f)$ are commutative. Moreover, $G(f) = S(f) \cap \mathrm{Aut}_X$ for every $f \in \mathrm{Aut}_X(k)$. We show that every bounded automorphism $f$ is a bounded endomorphism (as one would expect): since $\deg(f^n) =1$ for all $n$, we have $\deg|_{S(f)} = 1$ by continuity of the degree map and density of $\langle f \rangle$ in $S(f)$. Hence $S(f) \subseteq \mathrm{Aut}_X$ and $G(f) = S(f)$.

As an example, let $(X,f)$ be a dynamical system, with $f$ finite, of degree no less than two. Then the iterates $f^n$ belong to infinitely many connected components of $\mathrm{End}_X$, and $f$ is not bounded: indeed, $\deg(f^n) = \deg(f)^n \to \infty$ in the limit $n \to \infty$, and the degree remains the same on each connected component of $\mathrm{End}_X$. In particular, the bounded endomorphisms of a curve are exactly those of degree no greater than one, i.e. the constant endomorphisms as well as the automorphisms. See Lemmas \ref{falgqsf} and \ref{lemmaaut0} and Remark \ref{algnotclosedcomp} for more examples.

As we shall see shortly, the structure of $S(f)$ is closely related to that of a certain type of algebraic group known as monothetic. We use \cite{falcone2007monothetic} as a reference, noting that Definition \ref{defmonothetic} differs slightly from the (more general) definition introduced therein; however, our definition will be sufficient for the subsequent applications.

\begin{definition}\label{defmonothetic}
An algebraic group $G$ is called \emph{monothetic} if there exists $g \in G(k)$ such that the cyclic subgroup $\langle g \rangle$ of $G$ is schematically dense in $G$. The $k$-rational point $g$ is then called a \emph{generator} of $G$.
\end{definition}

For instance, if $f$ is a bounded automorphism, the algebraic group $G(f)$ is monothetic with generator $f$.

\begin{remark}\label{monosmooth}
It follows immediately that a monothetic algebraic group $G$ is commutative. Moreover, $G$ is geometrically reduced by \cite[Proposition 2.48]{milne2017groups}, and hence smooth by \cite[Proposition 1.26]{milne2017groups}. Finally, the structure of $G$ is given in the main theorem of \cite{falcone2007monothetic}.
\end{remark}

\begin{theorem}\label{mainthm0}
Let $(X,f)$ be a dynamical system, with $f$ bounded.
\begin{enumerate}[(1)]
    \item There exists a unique idempotent $e$ in $S(f)$.
    \item The subsemigroup scheme $eS(f) \subseteq S(f)$ is a closed algebraic subgroup of $\mathrm{End}_X$, denoted by $G$, whose neutral element is $e$. We have $f^n \in G$ for all $n \gg 1$.
    \item Let $m$ be the smallest positive integer such that $f^n \in G$ for all $n \geq m$. Then
    \begin{equation}\label{sfdecomp}
        S(f) = \{ f, f^2, \ldots , f^{m-1} \} \coprod G,
    \end{equation}
    where the iterates $f, f^2, \ldots , f^{m-1}$ are pairwise distinct. In particular, $G$ is a union of connected components of $S(f)$.
    \item The algebraic group $G$ is monothetic: we have 
    \[
    G = S(ef)  = \overline{\{ (ef)^\ell \mid \ell \geq 0 \}} = \overline{\{ (ef)^\ell \mid \ell \in \mathbb{Z} \}}.    
    \]
    In particular, $G$ is smooth.
    \item $G$ is the smallest closed algebraic subgroup of $\mathrm{End}_X$ containing $f^n$ for all $n \gg 1$.
\end{enumerate}
\end{theorem}

In this case, we call $G = eS(f)$ the monothetic group \emph{associated} with $f$ or with $(X,f)$.

\begin{theorem}\label{mainthm1}\leavevmode
\begin{enumerate}[(1)]
    \item Let $(X,f)$ be a dynamical system, with $f$ bounded, and let $e$ be the unique idempotent in $S(f)$. The algebraic semigroup $S(f)$ is uniquely determined by the triple $(G, g, m)$ consisting of the monothetic group $G= eS(f)$ associated with $f$, the $k$-rational generator $g=ef \in G(k)$, and the smallest positive integer $m$ such that $f^n \in G$ for all $n \geq m$.
    \item Conversely, given a triple $(H, h, p)$, where $H$ is a monothetic algebraic group with generator $h\in H(k)$, and $p$ is a positive integer, there exists a dynamical system $(X,f)$ (with $f$ bounded) such that $H$ contains $f^n$ for all $n \geq p$, with $p$ being the smallest positive integer having this property, and $H$ is the monothetic group associated to $f$.
\end{enumerate}
\end{theorem}

The proof of the above theorem, given in Section \ref{AE}, uses a Cayley-style result for algebraic groups: namely, that every algebraic group is isomorphic to a subgroup of the automorphism scheme $\mathrm{Aut}_X$ of some projective variety $X$; this is established in Proposition \ref{monoinaut}.

The next result describes the iterated image and the dynamical Stein factorization of a bounded endomorphism. First we recall that a morphism $r: X \to Y$ (in any category) is called a \emph{retraction} if there exists a morphism $s:Y \to X$ such that $r \circ s = \mathrm{id}_Y$, i.e. $r$ has a right inverse. Note that an idempotent endomorphism $e$ of $X$ is a retraction $r$ onto its image, i.e. $e$ restricts to the identity on its image: $r \circ i = \mathrm{id}_{e(X)}$ with $i: e(X) \to X$ being the inclusion. Conversely, given a retraction $r: X \to Y$ to a closed subvariety $i: Y \to X$, the composition $i \circ r$ is an idempotent endomorphism of $X$.

\begin{proposition}\label{mainprop2}
Let $(X,f)$ be a dynamical system, with $f$ bounded, and let $e$ be the unique idempotent in $S(f)$. Write $e = i\circ r$ with $r$ and $i$ as above, and set $Y=e(X)$. Let $m$ be the smallest positive integer such that $f^n \in G= eS(f)$ for all $n \geq m$.
\begin{enumerate}[(1)]
    \item We have $f^n(X)= Y$ for all $n\geq m$. In particular, $Y$ is the iterated image of $f$.
    \item For any $n \geq m$, the following diagram is the $n$th Stein factorization of $(X,f)$:
    \[
    \begin{tikzcd}[arrows={-stealth}]
    X \ar[rr, "f^n"] \ar[dd, "r"'] && X \\
    \\
    Y \ar[rr, "f^n_Y"] && Y \ar[uu, hook, "i"']
    \end{tikzcd}
    \]
    In particular, $Y$ is the iterated Stein factor of $(X,f)$.
    \item We have $m \leq \dim X$.
\end{enumerate}
\end{proposition}

Finally, we relate the boundedness of an endomorphism $f$ to that of the induced endomorphism $f_Y$ of the iterated image:
\begin{theorem}\label{mainthm2}
Let $(X,f)$ be a dynamical system, with iterated image $Y$. The following assertions are equivalent:
        \begin{enumerate}[(1)]
            \item $f$ is bounded.
            \item All iterates of $f$ are bounded.
            \item Some iterate of $f$ is bounded.
            \item The induced endomorphism $f_Y: Y \to Y$ is a bounded automorphism.
        \end{enumerate}
In this case, the monothetic group $G$ associated with $f$ is isomorphic to $G(f_Y)$; an algebraic subgroup of $\mathrm{Aut}_Y$. In particular, $\dim G \leq \dim \mathrm{Aut}_Y$.
\end{theorem}

\section{Proofs of Statements About Iterated Images}\label{proofsitim}
\subsection{Proof of Theorem \ref{itimage}}
We begin by recalling that for a projective variety $X$, the quotient of the free abelian group generated by curves on $X$, by numerical equivalence, is denoted by $\mathrm{N}_1(X)$. The group $\mathrm{N}_1(X)$ is free of finite rank; this rank is called the \emph{Picard number} of $X$. Every dominant morphism $f:X \to Y$ of projective varieties induces a $\mathbb{Q}$-linear map $f_*:\mathrm{N}_1(X)_{\mathbb{Q}} \to \mathrm{N}_1(Y)_{\mathbb{Q}}$, where $\mathrm{N}_1(X)_{\mathbb{Q}}:= \mathrm{N}_1(X) \otimes_{\mathbb{Z}} \mathbb{Q}$. We shall also need the following two lemmas:

\begin{lemma}\label{n1xbij}
If such a morphism $f$ is surjective, then so is the induced $\mathbb{Q}$-linear map $f_*$.
\end{lemma}
\begin{proof}
Let $C$ be a curve on $Y$, and consider the (surjective) restriction $f^{-1}(C) \to C$ of $f$. Then $\dim f^{-1}(C) \geq \dim C =1$. Denote by $X_1,\ldots,X_r$ the irreducible components of $f^{-1}(C)$, then $f$ has to surject some $X_i$ onto $C$. Choosing two distinct closed points $f(a)$ and $f(b)$ on $C$ with $a,b \in X_i$, there exists a curve $C'$ on $X_i$ that passes through $a$ and $b$ (by a special case of \cite[Corollary 1.9]{poonen2016bertini}). Hence $f(C')=C$, and $f_*C'=nC$ for some nonzero $n$. Passing to classes, we have $[C]=(1/n)f_*[C']=f_*[(1/n)C']$.
\end{proof}

\begin{lemma}\label{endofinite}
Let $(X,f)$ be a dynamical system. If $f$ is surjective, then it is finite.
\end{lemma}
We note that the converse is true and its proof is straightforward.
\begin{proof}
Suppose $f$ is surjective but not finite. Then $f$ contracts some curve $C$ on $X$ to a point. Thus $f_*[C]=0$, where we view $[C]$ as an element of the $\mathbb{Q}$-vector space $\mathrm{N}_1(X)_{\mathbb{Q}}$. On the other hand, we have $[C] \neq 0$ since $X$ is projective, contradicting the fact that $f_*:\mathrm{N}_1(X)_{\mathbb{Q}} \to \mathrm{N}_1(X)_{\mathbb{Q}}$ is an automorphism (following from Lemma \ref{n1xbij}).
\end{proof}

\begin{proof}[Proof of Theorem \ref{itimage}]
Suppose there exists a dynamical system $(X,f)$, with $X$ normal, whose iterated image is $Y$. Fix $n$ such that $f^n(X)=Y$. Consider the normalization $\nu:\tilde{Y} \to Y$, as well as the unique morphism $g:X\to \tilde{Y}$ such that $\nu \circ g = f^n$ (which exists by virtue of the universal property of $\nu$). We have the following commutative diagram, with $i:Y \to X$ being the inclusion:
\begin{center}
\begin{tikzcd}[arrows={-stealth}]
     &&     & Y \\
Y \ar[urrr, "{f^n_Y}" description, bend left] \ar[drrr, "{g\circ i}"' description, bend right] \ar[rr, "i", hook]    &&  X \ar[ur, "f^n"' description] \ar[dr, "g" description] & \\
     &&     & \widetilde{Y} \ar[uu, "\nu"']
\end{tikzcd}
\end{center}

We show that the morphism $g\circ i:Y \to \tilde{Y}$ is finite and surjective. Since $\nu$ is finite, the finiteness of $g\circ i$ is equivalent to that of $f^n_Y = f^n\circ i$  \cite[Proposition 6.1.5]{EGAII}. But the latter is surjective, hence finite by Lemma \ref{endofinite}. Surjectivity of $g\circ i$ follows directly, bearing in mind that $\dim Y = \dim \tilde{Y}$. Another way to prove that $g\circ i$ is surjective is noting that it is dominant (as $f^n_Y = \nu \circ (g\circ i)$ is surjective and $\nu$ is birational).

Conversely, suppose that $Y$ admits a finite morphism $h: Y \to \tilde{Y}$. Fix a closed immersion $j: Y \to \mathbb{P}^m_k$ and set $X=\tilde{Y}\times_k \mathbb{P}^m_k$. Then $X$ is a projective variety, normal by \cite[Proposition 6.14.1]{EGAIV2}. By the universal property of the fiber product $X$, there exists a unique morphism 
\[
\psi:Y\to X,\quad \psi(y) = (h(y),j(y)),
\]
which must be an immersion since $j$ is. In what follows, we view $Y$ as a subvariety of $X$ via $\psi$. We have an endomorphism 
\[
f:X \to X,\quad f(\tilde{y}, z) = (h(\nu(\tilde{y})), j(\nu(\tilde{y}))).
\]
We show the inclusions $f(X)\subseteq Y \subseteq f(Y) \subseteq f(X)$ (the last of which is clear), which will imply that $f(X)=Y$ and that $Y$ is the iterated image of $f$.
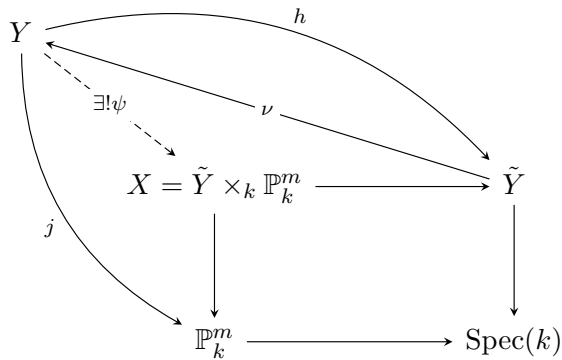
\[
\begin{tikzcd}[arrows={-stealth}]
    Y \ar[ddddr, bend right, "j"'] \ar[ddr, "\exists !\psi" description, dashed] \ar[ddrrr, bend left, "h"] &&& \\
    &&& \\    
    & X=\tilde{Y} \times_k \mathbb{P}^m_k \ar[rr] \ar[dd]&& \tilde{Y} \ar[uulll,"\nu" description] \ar[dd]\\
    &&& \\
    & \mathbb{P}^m_k \ar[rr] && \mathrm{Spec}(k)
\end{tikzcd}
\]
First, note that for any closed point $(\tilde{y}, z)\in X$ we have $f(\tilde{y}, z)=\psi(\nu(\tilde{y}))$, where $\nu(\tilde{y})\in Y$, yielding $f(X)\subseteq Y$. Next, let $y\in Y$, then $y= \nu (\tilde{y})$ for some $\tilde{y}\in \tilde{Y}$ (since $\nu$ is surjective), and $\tilde{y} = h(w)$ for some $w \in Y$ (since $h$ is surjective as well), thus $y=\nu(h(w))$ and
\[
\psi(y)=(h(y),j(y))=(h(\nu(h(w))), j(\nu(h(w))))=f(h(w),z),
\]
for any $z\in \mathbb{P}^m_k$. We can in particular write $\psi(y) = f(h(w), j(w))$, hence $\psi(y)=f(\psi(w))$ and $y\in f(Y)$.
\end{proof}

We conclude this subsection with the following natural questions posed by Lucy Moser-Jauslin: given a projective variety $X$, which closed subvarieties $Y$ of $X$ can occur as iterated images of endomorphisms of $X$? Does the condition that $Y$ admits a finite morphism $Y \to \tilde{Y}$ still characterize such iterated images?

\subsection{Proof of Proposition \ref{pinchprop}}\label{SCI}
We shall need a construction known as pinching, for which we use \cite{ferrand2003conducteur} as a reference. 

\begin{definition}\label{pinchdef}
Let $X$ be a scheme, let $Y$ be a closed subscheme of $X$, and let $g:Y \to Y'$ be an affine morphism of schemes. The amalgamated sum $X \amalg_{Y} Y'$ exists in the category of ringed spaces \cite[Paragraph 4.2 and Scolie 4.3]{ferrand2003conducteur}, forming a cocartesian square
\begin{equation}\label{cocart}
\begin{tikzcd}[arrows={-stealth}]
    Y \ar[r, "g"] \ar[d, hook, "i"']  & Y' \ar[d, "u"] \\
    X \ar[r, "p"]  & X \amalg_{Y} Y'
\end{tikzcd}
\end{equation}
If $X \amalg_{Y} Y'$ is a scheme, $p$ is affine, and $u$ is a closed immersion, then we call $X \amalg_{Y} Y'$ the \emph{pinching} of $X$ along $Y$ by $g$ . 
\end{definition}

\begin{theorem}\label{pinchingthm}
Let $X$ be a scheme, let $Y$ be a closed subscheme of $X$, and let $g:Y \to Y'$ be a finite morphism of schemes. Consider the cocartesian square \eqref{cocart}. If the schemes $X$ and $Y'$ satisfy the property that every finite subset of points is contained in an affine open, then 
\begin{enumerate}[(1)]
    \item the ringed space $X \amalg_{Y} Y'$ is a scheme;
    \item every finite subset of points of $X \amalg_{Y} Y'$ is contained in an affine open;
    \item the square \eqref{cocart} is cartesian;
    \item the morphism $p$ is finite;
    \item the morphism $u$ is a closed immersion;
    \item $p$ induces an isomorphism $X\setminus Y \to (X \amalg_{Y} Y') \setminus Y'$.
\end{enumerate}
\end{theorem}
In particular, $X \amalg_{Y} Y'$ is the pinching of $X$ along $Y$ by $g$.
\begin{proof}
See \cite[Th\'eor\`eme 5.4]{ferrand2003conducteur}.
\end{proof}

\begin{example}
Every variety is a pinching of its normalization. This follows from \cite[Th\'eor\`eme 5.1]{ferrand2003conducteur} together with the fact that the conductor can be defined globally as an ideal sheaf.
\end{example}

We recall the following notation and existence result, both of which will be used in the proof of Proposition \ref{pinchprop}. Let $E$ be an elliptic curve.
\begin{itemize}
    \item The multiplication-by-$n$ morphism $n_E: E \to E, x \mapsto nx$ is an isogeny of degree $n^2$, whose kernel is denoted by $E[n]$; a finite subgroup scheme of $E$.
    \item Let $F$ be a finite subgroup scheme of $E$. There exist a unique elliptic curve $E'$ and an isogeny $f: E \to E'$ with kernel $F$. The elliptic curve $E'$ is called the \emph{quotient} of $E$ by $F$, and denoted by $E/F$. See \cite[Proposition III.4.12]{silverman2009arithmetic}.
\end{itemize}

\begin{proof}[Proof of Proposition \ref{pinchprop}]
\begin{enumerate}[(1)]
    \item Consider the composition 
    \[
    \begin{tikzcd}[arrows={-stealth}]
    \tilde{Y} \ar[r, "\nu"] & Y \ar[r, "h"] & \tilde{Y}.
    \end{tikzcd}
    \]
    Then $\deg(h \circ \nu)=\deg(h) \cdot \deg(\nu) = \deg(h) \geq 2$ (where the last inequality follows from the assumption that $Y$ is singular). It follows that $g \leq 1$: indeed, if $g \geq 2$, then every nonconstant endomorphisms of $X$ is an automorphism by \cite[Example IV.2.5.4]{hartalggeom}.
    \item By projective Noether normalization, there exists a finite morphism $Y \to \mathbb{P}_k^1$. (This argument holds more generally for $k$ infinite and of any characteristic.)
    \item 
        \begin{enumerate}
            \item Choosing a $k$-rational point $0 \in \tilde{Y}(k)$ as origin, we get the elliptic curve $(\tilde{Y},0)$. Then we can write (in a unique way) $h\circ \nu = \tau_{z_0} \circ f$, where $\tau_{z_0}$ is the translation by some $z_0 \in \tilde{Y}(k)$, and $f$ is an endomorphism of $\tilde{Y}$ fixing $0$. Hence $(h \circ \nu)(z) = f(z) + z_0$ for all $z \in \tilde{Y}(k)$. 

            Next we claim that $h \circ \nu$ fixes some $y \in \tilde{Y}(k)$. Indeed, the equality $(h \circ \nu)(y) = y$ is equivalent to $f(y) - y = -z_0$, i.e. $(f - \mathrm{id}_{\tilde{Y}})(y)= -z_0$. Since 
            \[
            2 \leq \deg(h \circ \nu) = \deg(\tau_{z_0}) \cdot \deg(f) = \deg(f),
            \]
            we get $f \neq \mathrm{id}_{\tilde{Y}}$, so that $f- \mathrm{id}_{\tilde{Y}}$ is nontrivial, hence surjective. The claim follows. We have thus shown that $h \circ \nu$ is an isogeny of the elliptic curve $(\tilde{Y},y)$; its kernel is a finite subgroup scheme of $\tilde{Y}$.

            Finally, we show that for every $y \in Y$, there exists $z \in \tilde{Y}$ such that $\nu^{-1}(y) \subseteq z + \ker(h \circ \nu)$. Indeed, by surjectivity of $\nu$, we may choose $z \in \tilde{Y}$ with $\nu(z) = y$. For any $x \in \nu^{-1}(y)$, we have $\nu(x) = y$, so $(h \circ \nu)(x) = h(y) = (h \circ \nu)(z)$. It follows that $x - z \in \ker(h \circ \nu)$, and hence $x \in z + \ker(h \circ \nu)$. Since $x$ was arbitrary, the inclusion holds.
            \item If there exists a finite morphism $h:Y \to \tilde{Y}$, then it follows from the previous point $\text{(a)}$ that each $F_i$ is contained in some translate of $\ker(h \circ \nu)$. Conversely, the diagram 
                \[
                \begin{tikzcd}[arrows={-stealth}]
                \tilde{Y} & F_1 \coprod \cdots \coprod F_m  \ar[r] \ar[l, "i"'] & \mathrm{Spec}(k) \coprod \cdots \coprod \mathrm{Spec}(k)
                \end{tikzcd}
                \]
            (with $i$ being the inclusion, and the other arrow being induced by the structural morphisms $F_i \to \mathrm{Spec}(k)$) can be completed into a cocartesian square, by Theorem \ref{pinchingthm}, as follows
                \begin{equation}\label{pinch}
                \begin{tikzcd}[arrows={-stealth}]
                F_1 \coprod \cdots \coprod F_m \ar[rr] \ar[dd, "i"'] && \mathrm{Spec}(k) \coprod \cdots \coprod \mathrm{Spec}(k) \ar[dd]\\
                \\
                \tilde{Y} \ar[rr, "p"] && Y
                \end{tikzcd}    
                \end{equation}
            We note that the pinching morphism $p: \tilde{Y} \to Y$ coincides with the normalization $\nu$ of $Y$: indeed, $p$ is integral (being finite by Theorem \ref{pinchingthm}), and birational (by the same theorem, as $p$ induces an isomorphism between the open subschemes $\tilde{Y}\setminus \left( F_1 \coprod \cdots \coprod F_m \right)$ and $Y\setminus \left( \mathrm{Spec}(k) \coprod \cdots \coprod \mathrm{Spec}(k) \right)$.

            On the other hand, since $F$ is a finite subgroup of $\tilde{Y}(k)$, it must be $n$-torsion for some $n \geq 1$. In other words, $F \subseteq \tilde{Y}[n]$. We thus get a commutative triangle of isogenies
            \begin{equation*}
            \begin{tikzcd}[arrows={-stealth}]
            \tilde{Y} \ar[r] \ar[dr, "g"'] & \tilde{Y}/F  \ar[d, dashed]\\
            & \tilde{Y}/\tilde{Y}[n] \cong \tilde{Y}
            \end{tikzcd}    
            \end{equation*}
            Consider the morphisms $g: \tilde{Y} \to \tilde{Y}$ (in the above triangle) and $\coprod_{j=1}^m \mathrm{Spec}(k) \to \tilde{Y}$ arising from the fact that each inclusion $F_j \to \tilde{Y}$ factors through the structural morphism $F_j \to \mathrm{Spec}(k)$. By the universal property of the cocartesian square \eqref{pinch}, there exists a unique morphism $h: Y \to \tilde{Y}$ making the following diagram commute
            \[
            \begin{tikzcd}[arrows={-stealth}]
                \coprod_{j=1}^m F_j \ar[r] \ar[d, "i"'] & \coprod_{j=1}^m \mathrm{Spec}(k) \ar[d] \ar[dr] & \\
                \tilde{Y} \ar[r, "\nu" description] \ar[rr, bend right=30, "g" description] & Y \ar[r, "h" description, dashed] & \tilde{Y}/\tilde{Y}[n] \cong \tilde{Y}
            \end{tikzcd}
            \]
            Finally, since $g= h\circ \nu$ is quasi-finite, it follows directly that $h$ is quasi-finite, hence finite (a proper morphism is finite if and only if it is quasi-finite). \qedhere
            \end{enumerate}
\end{enumerate}
\end{proof}

\begin{example}\label{pinchingcounterex}
Suppose $\mathrm{char}(k)= p >0$, and let $\tilde{Y}$ be a smooth curve defined over $\mathbb{F}_{p^r}$ for some $r$; in other words, $\tilde{Y} \cong X \times_{\mathbb{F}_{p^r}} \mathrm{Spec}(k)$ for some smooth curve $X$ over $\mathbb{F}_{p^r}$. Let $Z$ be a fat point of $\tilde{Y}$ of degree no less than two such that $\mathscr{O}(Z)^{p^r} = k$. By Theorem \ref{pinchingthm}, we have a (nontrivial) pinching diagram 
\[
    \begin{tikzcd}[arrows={-stealth}]
    Z \ar[r, "\sigma"] \ar[d, "i"']  & \mathrm{Spec}(k) \ar[d] \\
    \tilde{Y} \ar[r, "p"]  & Y = \tilde{Y} \amalg_{Z} \mathrm{Spec}(k)
    \end{tikzcd}
\]
where $\sigma$ is the structural morphism of $Z$, $i$ is the inclusion, and $p$ is the pinching morphism (which coincides with the normaliztion $\nu: \tilde{Y} \to Y$). Consider the morphism $\mathrm{Spec}(k) \to \tilde{Y}$ through which $i$ factors, as well as the relative Frobenius morphism $F_{\tilde{Y}/k,r}: \tilde{Y} \to \tilde{Y}^{(p^r)}$. We have $\tilde{Y}^{(p^r)} \cong \tilde{Y}$. By the universal property of pinching, there exists a unique morphism $h: Y \to \tilde{Y}$ making the following diagram commute
\[
    \begin{tikzcd}[arrows={-stealth}]
    Z \ar[r, "\sigma"] \ar[d, "i"']  & \mathrm{Spec}(k) \ar[d] \ar[rd] & \\
    \tilde{Y} \ar[r, "p" description] \ar[rr, "F_{\tilde{Y}/k,r}"' description, bend right=30] & Y \ar[r, "h" description, dashed] & \tilde{Y} \cong \tilde{Y}^{(p^r)}
    \end{tikzcd}
\]
Finally, note that $h$ is finite.
\end{example}

\section{Proofs of Statements About Dynamical Stein Factorizations}\label{DSF}
\subsection{Proof of Proposition \ref{stability}}
\begin{proof}[\indent\unskip\nopunct]
Let $Y$ be the iterated image of $(X,f)$, and fix $n_0 \gg 1$ such that $f^{n_0}(X) = Y$. For any $m \geq 1$, consider the following commutative diagram
\begin{equation}\label{steincd}
    \begin{tikzcd}[arrows={-stealth}]
    X \ar[rr, "f^{n_0}"] \ar[dd, "g_{n_0}"'] && X \ar[rr, "f^m"]  && X \\
    \\
    Z_{n_0} \ar[rr, "h'_{n_0}"] && f^{n_0}(X) = Y \ar[uu, "i"', hook] \ar[rr, "f^m_Y"] && f^{n_0+m}(X) = Y \ar[uu, "i"', hook]
    \end{tikzcd}
\end{equation}
where the left-hand square is the Stein factorization of $f^{n_0}$. Note that $f^m_Y$ is finite by Lemma \ref{fny}, hence $f^m_Y \circ h'_{n_0}$ is finite as well. Therefore, the outer square in \eqref{steincd} is the Stein factorization of $f^{n_0+m}$: we have for all $m \geq 1$
\[
g_{n_0+m} = g_{n_0}, \quad Z_{n_0+m} = Z_{n_0}, \quad h'_{n_0+m} = f^m_Y \circ h'_{n_0}. 
\] 
This shows that the sequences $(g_n)$ and $(Z_n)$ stabilize.
\end{proof}

\begin{remark}\label{nsteinnimage}
Let $(X,f)$ be a dynamical system, and consider its sequence of Stein factorizations \eqref{nstein2}. Denote by $N_{\mathrm{image}}$ and $N_{\mathrm{Stein}}$ the smallest positive integers at which the sequences $(f^n(X))$ and $(Z_n)$ stabilize, respectively. Then 
\[
N_{\mathrm{image}}  = N_{\mathrm{Stein}} \leq \dim X.
\]
Indeed, we have just shown that once $(f^n(X))$ stabilizes, so does $(Z_n)$, implying $N_{\mathrm{Stein}} \leq N_{\mathrm{image}}$. Conversely, if $(Z_n)$ stabilizes, then the equality $f^n(X) = h_n(Z_n)$ implies that $(f^n(X))$ stabilizes as well, hence $N_{\mathrm{image}} \leq N_{\mathrm{Stein}}$. Finally, we see from the following chain of closed subvarieties of $X$
\[
X \supsetneqq f(X) \supsetneqq f^2(X) \supsetneqq  \cdots \supsetneqq f^{N_{\mathrm{image}}}(X)
\]
that $N_{\mathrm{image}} \leq \dim X$.
\end{remark}

\subsection{Proof of Proposition \ref{steinfactoriteratedimage}}
\begin{proof}[\indent\unskip\nopunct]
We will show that both statements of Proposition \ref{steinfactoriteratedimage} hold for any $n \geq N_{\mathrm{Stein}}$.
\begin{enumerate}[(1)]
    \item Consider the following commutative diagram
    \begin{equation*}
    \begin{tikzcd}[arrows={-stealth}]
        X \ar[rrrrdd, bend left=20, "f^n" description] \ar[dd, "g_n = g_{n+1}"'] &&  && \\
        \\
        Z_n = Z_{n+1}=Z \ar[rr, "h'_n"] \ar[ddrr, "h'_{n+1}"'] && f^n(X) =Y \ar[rr, hook] \ar[dd, "f_Y"] && X \ar[dd,"f"] \\
        \\
        && f^{n+1}(X) = Y \ar[rr, hook] && X
    \end{tikzcd}
    \end{equation*}
    combining the Stein factorizations of $f^n$ and $f^{n+1}$, with $h'_n$ and $h'_{n+1}$ finite and surjective. The bottom left triangle yields $\deg(h'_{n+1}) = \deg(f_Y) \cdot \deg(h'_n)$.
    \item Concatenating two copies of the $n$th Stein factorization of $(X,f)$ as follows
    \[
    \begin{tikzcd}[arrows={-stealth}]
    X \ar[rr, "f^n"] \ar[d, "g_n"'] && X \ar[rr, "f^n"] \ar[dr, "g_n"']  && X \\
    Z \ar[r, "h'_n"]   & Y \ar[ur, "i"'] && Z \ar[r, "h'_n"] & Y \ar[u, "i"']
    \end{tikzcd}
    \]
    we get the commutative diagram
    \[
    \begin{tikzcd}[arrows={-stealth}]
    Y \ar[rr, "f_Y^n"] \ar[dr, "g_n\circ i"'] && Y \ar[rr, "f_Y^n"] \ar[dr, "g_n\circ i"']  && Y \\
    & Z \ar[ur, "h'_n"']  && Z \ar[ur, "h'_n"'] &
    \end{tikzcd}
    \]
    which in turn yields the desired commutative square with $\varphi_n = g_n\circ i \circ h'_n$. \qedhere
\end{enumerate}
\end{proof}

\begin{remark}
Let $(X,f)$ be a dynamical system with $n$th Stein factorization \eqref{nstein2}, iterated Stein factor $Z$, and iterated image $Y$. If $X$ is normal, then so is $Z_n$ (this follows from $g_n$ being a contraction, i.e. $(g_n)_*\mathscr{O}_X = \mathscr{O}_{Z_n}$). However, $Y$ may fail to be normal in this case (see Proposition \ref{pinchprop} for example). It might be worth noting that for $n \geq N_{\mathrm{Stein}}$, the finite morphism $h'_n:Z \to Y$ is almost never the normalization morphism of $Y$ since $h'_n$ is rarely birational by the first point of Proposition \ref{steinfactoriteratedimage}.
\end{remark}

\section{Proofs of Statements About Bounded Endomorphisms}\label{AE}
We record some basic results about bounded endomorphisms, before moving to the main results.
\begin{lemma}\label{falgqsf}
Suppose $k$ is separably closed. Let $(X,f)$ be a dynamical system, and consider the canonical projections $q_{\mathrm{End}_X}: \mathrm{End}_X \to \pi_0\mathrm{End}_X$ and $q_{S(f)}: S(f) \to \pi_0(S(f))$ defined as in \cite[I,\S 4,6.5]{demazure1970groupes}.
\begin{enumerate}[(1)]
    \item We have $q_{\mathrm{End}_X}(S(f)) = \langle q_{\mathrm{End}_X}(f) \rangle = \{ q_{\mathrm{End}_X}(f)^n \mid n\geq 1 \}$.
    \item Similarly, we have $\pi_0 (S(f)) = \langle q_{S(f)}(f) \rangle = \{ q_{S(f)}(f)^n \mid n\geq 1 \}$.
    \item The following statements are equivalent:
    \begin{enumerate}
        \item $f$ is bounded.
        \item $q_{\mathrm{End}_X}(S(f))$ is finite.
        \item $\pi_0 (S(f))$ is finite.
    \end{enumerate}
    In this case, we have $|\pi_0 (S(f))| \geq |q_{\mathrm{End}_X}(S(f))|$.
    \item All finite cyclic subsemigroups of the monoid $(\pi_0\mathrm{End}_X)(k)$ are of the form $\langle q_{\mathrm{End}_X}(g) \rangle$ for some bounded $g \in \mathrm{End}_X(k)$.
\end{enumerate}
\end{lemma}
\begin{proof}
\begin{enumerate}[(1)]
    \item As $\langle f \rangle$ is dense in $S(f)$, the image $q_{\mathrm{End}_X}(\langle f \rangle) = \langle q_{\mathrm{End}_X}(f) \rangle$ is dense in $q_{\mathrm{End}_X}(S(f))$, which has the discrete topology. The result follows.
    \item This is done exactly as in the previous point, bearing in mind that $\pi_0 (S(f)) = q_{S(f)}(S(f))$.
    \item We first show $(a)\Leftrightarrow (b)$. If $S(f)$ is of finite type, it has finitely many connected components, so that $q_{\mathrm{End}_X}(S(f))$ is finite. Conversely, if $q_{\mathrm{End}_X}(S(f)) = \{ a_1, \ldots, a_m\}$, then $q_{\mathrm{End}_X}^{-1}(q_{\mathrm{End}_X}(S(f))) = \bigcup_{i=1}^m q_{\mathrm{End}_X}^{-1}(a_i)$ is a union of connected components of $\mathrm{End}_X$ and contains $S(f)$, hence the latter is of finite type. 
    
    To show $(c)\Rightarrow (b)$, consider the inclusion $i: S(f) \to \mathrm{End}_X$; we have
    \begin{equation}\label{qsf}
        q_{\mathrm{End}_X}(S(f)) = q_{\mathrm{End}_X}(i(S(f))) = (\pi_0 i)(q_{S(f)}(S(f))) = (\pi_0 i)(\pi_0 (S(f))).
    \end{equation}
    Hence, if $\pi_0(S(f))$ is finite, then so is $q_{\mathrm{End}_X}(S(f))$. 
    \[
    \begin{tikzcd}[arrows={-stealth}]
    S(f) \ar[r, "i"] \ar[d, "q_{S(f)}"'] &  \mathrm{End}_X \ar[d, "q_{\mathrm{End}_X}"] \\
     \pi_0 (S(f)) \ar[r, "\pi_0 i"]   & \pi_0\mathrm{End}_X
    \end{tikzcd}
    \]    
    Finally, we show $(a)\Rightarrow (c)$. If $S(f)$ is of finite type, then the induced morphism $\pi_0i: \pi_0(S(f)) \to \pi_0\mathrm{End}_X$ has finite fibers (i.e. $\pi_0i$ is finite-to-one): each fiber $F$ of $q_{\mathrm{End}_X}$ is a connected component of $\mathrm{End}_X$, hence $S(f) \cap F$ has finitely many connected components, and $q_{S(f)}(S(f) \cap F)$ is finite. It follows that $\pi_0(S(f))$ is finite by \eqref{qsf} since $q_{\mathrm{End}_X}(S(f))$ is finite.
    \item Consider a finite subsemigroup $\langle a \rangle$ of $(\pi_0\mathrm{End}_X)(k)$. Then $a = q_{\mathrm{End}_X}(g)$ for some $g \in \mathrm{End}_X(k)$, and, by the previous point, $g$ must be bounded.   \qedhere
\end{enumerate}
\end{proof}

\begin{lemma}\label{lemmaaut0}
Let $(X,f)$ be a dynamical system.
\begin{enumerate}[(1)]
    \item If $f \in \mathrm{Aut}_X^0(k)$, then $f$ is bounded.
    \item Conversely, if $f \in \mathrm{Aut}_X(k)$ is bounded, then some iterate of $f$ belongs to $\mathrm{Aut}_X^0(k)$. 
\end{enumerate}
\end{lemma}
\begin{proof}
\begin{enumerate}[(1)]
    \item We have $G(f) \subseteq \mathrm{Aut}^0_X$ as $\mathrm{Aut}^0_X$ is closed, hence $f$ is bounded. See Example \ref{torsionendo}.
    \item Recall that $q_{\mathrm{Aut}_X}$ is the restriction of $q_{\mathrm{End}_X}$ to $\mathrm{Aut}_X$. The image $\overline{f} = q_{\mathrm{End}_X}(f)$ has finite order in $\pi_0\mathrm{Aut}_X$, hence there exists $m \geq 1$ such that
    \[
    q_{\mathrm{End}_X}(f^m) = \overline{f^m} = \overline{f}^m = \overline{\mathrm{id}_X} = q_{\mathrm{End}_X}(\mathrm{id}_X), 
    \]
    implying $f^m \in \mathrm{Aut}_X^0(k)$. \qedhere
\end{enumerate}
\end{proof}

\begin{remark}\label{algnotclosedcomp}
The set of all bounded endomorphisms of a projective variety $X$ is generally not closed under composition. For instance, consider the abelian surface $X = E \times_k E$; the square of an elliptic curve $E$, then $(\pi_0\mathrm{Aut}_X)(k)$ contains a copy of $\mathrm{GL}_2(\mathbb{Z})$: indeed, the following action of $\mathrm{GL}_2(\mathbb{Z})$ on $X$ 
\[
\begin{pmatrix}
  a & b\\
  c & d
\end{pmatrix}\cdot \begin{pmatrix}
  x \\
  y
\end{pmatrix} = 
\begin{pmatrix}
  ax+ by \\
  cx+ dy
\end{pmatrix}
\]
gives rise to automorphisms of $X$ fixing $0_X= (0_E,0_E)$. Now there exist elements $f,g \in \mathrm{GL}_2(\mathbb{Z})$ of finite order such that $fg$ has infinite order. For example, taking 
\[
f= \begin{pmatrix}
  -1 & 1\\
  0 & 1
\end{pmatrix} \quad \text{ and } \quad
g= \begin{pmatrix}
  -1 & 0\\
  0 & 1
\end{pmatrix},
\]
we get $f^2 = g^2 = \mathrm{id}$ and 
$(fg)^n = \left(\begin{smallmatrix}
  1 & n\\
  0 & 1
\end{smallmatrix}\right)$
for every $n$. On the other hand, it is straightforward that the set of bounded endomorphisms of $X$ is closed under conjugation by automorphisms. Finally, Theorem \ref{mainthm2} implies that the set of bounded endomorphisms of $X$ is closed under taking powers and roots.
\end{remark}

\subsection{Proof of Theorem \ref{mainthm0}}
The argument uses an intermediate result that we record in the following lemma, the proof of which can be found in \cite[Proposition 3]{brion2014automorphisms}; we provide it here for the sake of completeness.
\begin{lemma}\label{mainlemma1}
Let $(X,f)$ be a dynamical system, with $f$ bounded.
\begin{enumerate}[(1)]
    \item The family $\Phi = \{ S(f^n) \mid n \geq 1 \}$ of closed subsemigroups of $S(f)$ has a smallest (i.e. least) element, which we denote by $S(f^{n_0})$: for all $n$, we have $S(f^{n_0}) \subseteq S(f^n)$.
    \item The semigroup $S(f^{n_0})$ is a connected algebraic group.
\end{enumerate}
\end{lemma}
\begin{proof}
\begin{enumerate}[(1)]
    \item By noetherianity of $S(f)$ (following from $f$ being bounded), the family $\Phi$ has a minimal element $S(f^{n_0})$. As $\Phi$ is down-directed: $S(f^m) \cap S(f^n) \supseteq S(f^{mn})$ for all $m,n \geq 1$, we see that $S(f^{n_0})$ is the smallest element in $\Phi$.
    \item Consider the canonical projection 
    \[
    q=q_{S(f^{n_0})}: S(f^{n_0}) \to \pi_0 S(f^{n_0}),
    \]
    and set $a = q(f^{n_0})$. Then $q(S(f^{n_0}))= \langle a \rangle$; a finite semigroup. Let $a^h = q(f^{n_0h})$ be the unique idempotent element in $\langle a \rangle$. Then $q^{-1}(a^h)$ is a closed connected subsemigroup of $S(f^{n_0})$, and contains $f^{n_0h}$. We thus have $S(f^{n_0h}) \subseteq q^{-1}(a^h) \subseteq S(f^{n_0})$, implying, by minimality of $S(f^{n_0})$, that 
    \[
    S(f^{n_0h}) = q^{-1}(a^h) = S(f^{n_0}).
    \]
    In particular, $S(f^{n_0})$ is connected. 
    
    We show that $S(f^{n_0})$ is a group: by the first point of \cite[Proposition 2]{brion2014automorphisms}, since $S(f^{n_0})$ is commutative, it must contain a unique idempotent, denote it by $e$. By the second point of the same proposition, $S(f^{n_0})$ is isomorphic to its closed subsemigroup $eS(f^{n_0})$, and the latter is a group scheme.  \qedhere
\end{enumerate}
\end{proof}

We now turn to the proof of Theorem \ref{mainthm0}.

\begin{proof}[Proof of Theorem \ref{mainthm0}]
\begin{enumerate}[(1)]
    \item Using the notation of Lemma \ref{mainlemma1}, we show that the neutral element $e$ in $S(f^{n_0})$ is the only idempotent in $S(f)$. If $e_1$ is idempotent in $S(f)$, then $e_1 =e_1^{n_0} \in S(f^{n_0})$ is idempotent in $S(f^{n_0})$, hence $e_1 =e$.
    \item As $S(f)$ is commutative, we have $eS(f) = eS(f)e$; this is a closed submonoid of $S(f)$ with neutral element $e$ and, by the previous point, has no other idempotents. Hence $eS(f)$ is a closed algebraic subgroup of $S(f)$ by \cite[Proposition 3(iii)]{brion2014algebraic}, which we denote by $G$. For any $\ell \geq 0$, we have 
    \begin{equation}\label{fing}
        f^{n_0 + \ell} = f^{n_0} f^\ell = (ef^{n_0}) f^\ell = f^{n_0} (ef^\ell) =f^{n_0} (ef)^\ell  \in G
    \end{equation}
    since $f^{n_0} = ef^{n_0} \in G$ and $ef \in G$. In other words, $f^n \in G$ for all $n \geq n_0$.
    \item We have 
    \begin{align*}
        S(f)  = \overline{\{ f^n \mid n\geq 1 \}} & = \overline{\{ f, f^2, \ldots , f^{m-1} \}} \cup \overline{\{ f^n \mid n\geq m\}} \\
        & = \{ f, f^2, \ldots , f^{m-1} \} \cup \overline{\{ f^n \mid n\geq m\}} \\
        & \subseteq \{ f, f^2, \ldots , f^{m-1} \} \cup G \subseteq S(f).
    \end{align*}
    Hence $S(f) = \{ f, f^2, \ldots , f^{m-1} \} \cup G$. Note that $\{ f, f^2, \ldots , f^{m-1} \}$ and $G$ are disjoint: if $f^i \in G$ for some $i \leq m-1$, then $f^{i+\ell} \in G$ for all $\ell \geq 0$ as in \eqref{fing}, contradicting the choice of $m$. Also note that the iterates $f, f^2, \ldots , f^{m-1}$ are pairwise distinct: if $f^i = f^j$ for some $1 \leq i < j \leq m-1$, then 
    \[
    G \ni f^m = f^{j+(m-j)} = f^j f^{m-j} = f^i f^{m-j} = f^{m-(j-i)} \in \{ f, f^2, \ldots , f^{m-1} \},
    \]
    contradicting the aforementioned disjointness.
    \item Again, just as in \eqref{fing}, we have for all $\ell \geq 0$ that $f^{m+\ell} = f^m (ef)^\ell$. Thus
    \[
    G = \overline{\{ f^{m+\ell} \mid \ell \geq 0\}} = \overline{\{ f^m (ef)^\ell \mid \ell \geq 0\}} = \overline{\{ (ef)^\ell \mid \ell \geq 0\}},
    \]
    where the last equality follows from $f^m \langle ef \rangle = \langle ef \rangle$.
    Moreover, we show that 
    \[
    G = \overline{\{ (ef)^\ell \mid \ell \in \mathbb{Z} \}}.
    \]
    Indeed, as $\overline{\{ (ef)^\ell \mid \ell \geq 0 \}}$ is a closed subsemigroup scheme of the monothetic algebraic group $\overline{\{ (ef)^\ell \mid \ell \in \mathbb{Z} \}}$, it is a closed \emph{subgroup} scheme of $\overline{\{ (ef)^\ell \mid \ell \in \mathbb{Z} \}}$ by \cite[Lemma 9]{brion2014automorphisms}, and since the former also contains the generator $g$, we must have equality. Smoothness of $G$ follows from Remark \ref{monosmooth}. 
    \item Let $H$ be a closed subgroup of $\mathrm{End}_X$ such that $f^n \in H$ for all $n \geq n_1$ (with $n_1$ a positive integer), then $H \supseteq S(f^{n_1})$. Additionally, $S(f^{n_1}) \supseteq S(f^{n_0})$ by minimality of the latter. In particular, the neutral element of $H$ is $e$. If $h$ is the inverse of $f^{n_1}$ in $H$, then $hf^{n_1} = e$, which implies $hf^{n_1+1} = ef \in H$, hence $G \subseteq H$ since $ef$ is a generator of $G$ by the previous point. \qedhere
\end{enumerate}
\end{proof}

\subsection{Proof of Theorem \ref{mainthm1}}
We shall use the following lemma in the proof of Proposition \ref{monoinaut}.
\begin{lemma}\label{alggrplemma}
Let $G$ be an algebraic group. There exists an affine closed subgroup $H$ of $G$ satisfying the following properties:
\begin{enumerate}[(1)]
    \item $G = H \cdot G_{\mathrm{ant}}$, where $G_{\mathrm{ant}}$ is the largest anti-affine (i.e. $\mathscr{O}(G_{\mathrm{ant}})=k$) closed subgroup of $G$.
    \item $H$ contains a normal subgroup $N$ of $G$ such that $G/N$ is proper.
\end{enumerate}
\end{lemma}
\begin{proof}
See \cite[Theorem 4.7]{brion2015extensions}. 
\end{proof}

The following result is well known for $G$ an affine algebraic group or an abelian variety, but the general case seems to be unrecorded.

\begin{proposition}\label{monoinaut}
Let $G$ be an algebraic group. There exists a projective variety $X$ on which $G$ acts faithfully. 
\end{proposition}
\begin{proof}
We begin by writing $G = H \cdot G_{\mathrm{ant}}$ as in the first point of Lemma \ref{alggrplemma}. We show that the quotient $G/H$ is an abelian variety. By the second point of Lemma \ref{alggrplemma}, $G/H$ is proper (since $G/N$ is proper, and $G/H$ is the image of $G/N$ by the morphism $G/N \to G/H$). Moreover, as 
\[
G/H = H\cdot G_{\mathrm{ant}}/H = G_{\mathrm{ant}} / (G_{\mathrm{ant}} \cap H)
\]
and $G_{\mathrm{ant}}$ is commutative, smooth, and connected (by \cite[Corollary 8.14, Proposition 8.37]{milne2017groups}), the quotient $G/H$ has a structure of (commutative) smooth (by \cite[Corollary 5.26]{milne2017groups}) connected algebraic group. Since $G/H$ is proper, it is an abelian variety. In particular, $G/H$ is projective, and the quotient morphism $G \to G/H$ is an $H$-torsor \cite[Definition 2.66]{milne2017groups}.

On the other hand, being affine, $H$ is isomorphic to a closed subgroup of some general linear group $\mathrm{GL}_n$, which in turn is a closed subgroup of $\mathrm{PGL}_{n+1} = \mathrm{Aut}_{\mathbb{P}_k^n}$. We thus get a faithful action of $H$ on $\mathbb{P}_k^n$. Let $H$ act on $G$ by $h \cdot g = h^{-1}g$ and on $G \times_k \mathbb{P}_k^n$ by $h \cdot (g,y) = (h^{-1}g , hy)$. Then the projection $\mathrm{pr}_1: G \times_k \mathbb{P}_k^n \to G$ is equivariant with respect to these actions. It follows from \cite[Proposition 7.1]{mumford1965git} that there exists a quotient $X = (G \times_k \mathbb{P}_k^n)/H$ and this quotient is a variety. Moreover, $X$ is projective over $G/H$, and the latter is projective over $k$.

Finally, $G$ acts faithfully on $X$ via multiplication on itself.
\end{proof}

\begin{proof}[Proof of Theorem \ref{mainthm1}]
\begin{enumerate}[(1)]  
    \item As shown in Theorem \ref{mainthm0}, we have
    \[
    S(f) = \{ f, \ldots, f^{m-1} \} \coprod G = \{ f, \ldots, f^{m-1} \} \coprod \overline{\{ (ef)^\ell \mid \ell \geq 0 \}}.
    \]
    The multiplication of $S(f)$ is given as follows: for any $a,b \in S(f)$, we have
    \begin{itemize}
        \item If $a,b \in G$, then $ab \in G$, i.e. multiplication is performed within $G$.
        \item If $a = f^i$, with $1 \leq i \leq m-1$, and $b \in G$, then 
        \begin{equation}\label{fatimesh}
            ab= f^ib = f^i(eb) = (ef^i)b = (ef)^ib.
        \end{equation}
        \item If $a = f^i$ and $b = f^j$ with $1 \leq i,j \leq m-1$, then $ab =f^{i+j}$ whenever $i+j \leq m-1$, otherwise we can write $i+j = qm+r$ with $q \geq 0$ and $0\leq r \leq m-1$, then 
        \[
        ab= f^r f^{qm} = f^r (e f^{qm}) = f^r (ef)^{qm} = (ef)^{r+qm} = (ef)^{i+j},
        \]
        where the second to last equality follows from \eqref{fatimesh}.
    \end{itemize}
    \item By Proposition \ref{monoinaut}, there exists a projective variety $Y$ such that $\mathrm{Aut}_Y \supseteq H$. Consider the projective variety $X = (\mathbb{P}_k^1)^p \times_k Y$, and define the endomorphism
        \[
        f: X \to X, \quad f(x_1, \ldots, x_p, y)=(0, x_1, \ldots, x_{p-1}, h(y)).
        \]
    Then $f^p(x_1, \ldots, x_p, y)=(0, \ldots, 0, h^p(y))$, and $f^n \in H$ for all $n \geq p$, which implies that $f$ is bounded. Evidently, $p$ is minimal. To show the last assertion, consider the endomorphism
    \[
    e: X \to X, \quad e(x_1, \ldots, x_p, y) = (0, \ldots, 0, y).
    \]
    Then $e$ is idempotent, and is the neutral element in $eS(f)$. (One can check by hand that $ef^n = f^n = f^ne$ for all $n \geq p$.) Finally, as $ef = h$ we have by virtue of the fourth point of Theorem \ref{mainthm0} that $eS(f) = H$.
    \qedhere
\end{enumerate}
\end{proof}

\subsection{Proof of Proposition \ref{mainprop2}}
\begin{proof}[\unskip\nopunct]
\begin{enumerate}[(1)]
    \item Recall from Theorem \ref{mainthm0} that $G = eS(f)$ is an algebraic group with neutral element $e$. Let $n\geq m$. Then $f^n \in G$. Therefore, $f^n = ef^n$, yielding $f^n(X) \subseteq Y$. Moreover, $f^n$ possesses an inverse $g=g(n) \in G$, so that $e = f^ng$, and hence $Y \subseteq f^n(X)$.
    \item This follows from 
    \[
    f^n = f^n e = (f^n i) r = (i f_Y^n) r,
    \]
    knowing that $r$ is a contraction (being a retraction) and $i \circ f^n_Y$ is finite.
    \item We will show that
    \[
    m = N_{\mathrm{image}} =  N_{\mathrm{Stein}} \leq \dim X.
    \]
    First, $ef^{N_{\mathrm{image}}} = f^{N_{\mathrm{image}}}$ as $e$ is the identity on $Y = f^{N_{\mathrm{image}}}(X)$, yielding $f^{N_{\mathrm{image}}} \in eS(f) =G$, hence $m \leq N_{\mathrm{image}}$ by the choice of $m$. Next, we have $N_{\mathrm{image}} = N_{\mathrm{Stein}}$ by Remark \ref{nsteinnimage}, and $N_{\mathrm{Stein}} \leq m$ by the previous point. Finally, $N_{\mathrm{image}} \leq \dim X$ by Remark \ref{nsteinnimage} once again.     \qedhere
\end{enumerate}
\end{proof}

\subsection{Proof of Theorem \ref{mainthm2}}
\begin{proof}[\indent\unskip\nopunct]
The implication $1 \Rightarrow 2$ follows from $S(f^n) \subseteq S(f)$, and $2 \Rightarrow 3$ is common sense. We show $3 \Rightarrow 1$: suppose $f^n$ is bounded for some $n$. By the third point of Theorem \ref{mainthm0}, we have a decomposition
\[
S(f^n) = \{ f^n, f^{2n}, \ldots ,f^{(p-1)n} \} \coprod H,
\]
with $H$ the monothetic group associated with $f^n$, and $p$ the smallest positive integer such that $(f^n)^k = f^{nk} \in H$ for all $k \geq p$. Let $\ell \ge 1$, and write $\ell = (np)q + r$ with $0 \leq q$ and $0 \leq r < np$, then $f^{\ell} = (f^{np})^q f^r$.
    \begin{itemize}
        \item If $q =0$, then $(f^{np})^q = \mathrm{id}_X$ and $f^{\ell} \in \bigcup_{r=0}^{np-1} \{f^r\}$.
        \item If $q >0$, then $(f^{np})^q \in H$ and $f^{\ell} \in \bigcup_{r=0}^{np-1} f^r H$. 
    \end{itemize}
    Hence for every $\ell \geq 1$, we have
    \[
    f^{\ell} \in \bigcup_{r=0}^{np-1} \{f^r\} \cup \bigcup_{r=0}^{np-1} f^r H 
    \]
    and $f$ is bounded.

For $1 \Rightarrow 4$, let $e$ be the unique idempotent in $S(f)$. Then $Y=e(X)$ by the first point of Proposition \ref{mainprop2}. Write $e = i\circ r$, with $r:X \to Y$ a retraction and $i:Y \to X$ the inclusion, and let $m$ be the smallest positive integer such that $f^n \in G= eS(f)$ for all $n \geq m$.

Let us first show that $f_Y$ is an automorphism. Let $n > m$, and let $g \in G$ be the inverse of $f^n$, i.e. $f^n g = g f^n= e$. Note that $f^n$ stabilizes $Y$ (by Lemma \ref{fny}), and so does $g$: we have $Y= e(X) = g(f^n(X))=g(Y)$. Restricting to $Y$, we get 
\[
f^n_Y g_Y = g_Y f^n_Y = \mathrm{id}_Y,
\]
so that $f^n_Y$ is an automorphism. Moreover, $f_Y$ is an automorphism since $\mathrm{id}_Y = f_Y(f_Y^{n-1}g_Y)$ and $\mathrm{id}_Y = (g_Y f^{n-1}_Y)f_Y = (f^{n-1}_Y g_Y)f_Y$, where the last equality follows from the fact that $f^{n-1}, g \in G$ commute, and hence so do their restrictions to $Y$.
    \[
    \begin{tikzcd}[arrows={-stealth}]
    X \ar[rr, "g"] \ar[rrrr, "e"' description, bend left]  && X  \ar[rr, "f^n"] && X \\
    &&   &&\\
    Y \ar[uu, hook, "i"] \ar[rr, "g_Y"'] \ar[rrrr, "{\mathrm{id}_Y}"' description, bend right] && Y \ar[uu, hook, "i"] \ar[rr, "f^n_Y"'] && Y \ar[uu, hook, "i"']
\end{tikzcd}
    \]
    
Now we show that $f_Y$ is bounded. Consider the morphism of monoid schemes 
    \[
    \theta: e\mathrm{End}_X e \to \mathrm{End}_Y, \quad g = eg'e \mapsto r g i = rg'i.
    \]
Then $\theta$ is an isomprphism: its inverse is given by $h \mapsto ihr$. As $S(f)$ is commutative, $ef = efe$, so that
    \[
    \theta (ef) = \theta (efe) = refei = rfi = f_Y.
    \]
Moreover, for any $n \geq m$, we have $ef^ne = ef^n = f^n$ (the last equality holds since $f^n \in G$ and $e$ is the neutral element in $G$). Then
    \[
    \theta(f^n) = \theta ((ef)^n) = \theta(ef)^n = (f_Y)^n = f_Y^n.
    \]
Note that $\theta$ restricts to an isomorphism $(e \mathrm{End}_X e)^{\times} \to \mathrm{Aut}_Y$ between the respective unit group schemes. Also, the monothetic group $G$ is an algebraic subgroup of $(e \mathrm{End}_X e)^{\times}$: indeed, since $\theta(ef) = f_Y$ is an automorphism, we have $ef \in (e \mathrm{End}_X e)^{\times}$, hence $G \subseteq (e \mathrm{End}_X e)^{\times}$. Finally, since $f^n \in G$ for all $n \geq m$, we get $f^n_Y = \theta(f^n) \in \theta(G) \subseteq \mathrm{Aut}_Y$, where $\theta(G) \cong G$; an \emph{algebraic} group, proving $f_Y$ is bounded.

We show $4 \Rightarrow 3$. Note that $f_Y^\ell \in \mathrm{Aut}_Y^0(k)$ for some $\ell \geq 1$ by the second point of Lemma \ref{lemmaaut0}. Consider the commutative triangle
\[
\begin{tikzcd}[arrows={-stealth}]
    X \ar[rr,"f^{\ell}"] \ar[drr,"{f^{2\ell}}"'] & & Y \ar[d, "f_Y^{\ell}"] \\
    &   & Y
\end{tikzcd}
\]
Then $f^\ell$ and $f^{2\ell}$ belong to the same connected component in $\mathrm{End}_X$. Hence all the iterates $f^{n\ell}$ belong to the same connected component in $\mathrm{End}_X$. In particular, $f^\ell$ is bounded.

For the last assertion, note that 
\begin{align*}
G(f_Y) & = \overline{\{ f_Y^n \mid n \geq 0 \}} \\ 
   & = \overline{\{ \theta(ef)^n \mid n \geq 0 \}} \cong \overline{\{ (ef)^n \mid n \geq 0 \}} = S(ef) = G. 
\end{align*}
\end{proof}

\begin{remark}
Let $(X,f)$ be a dynamical system, with $f$ bounded of iterated image $Y$. By Theorem \ref{mainthm2}, we have $\deg(f_Y)=1$, so Proposition \ref{steinfactoriteratedimage} implies that the numerical sequence $(\deg(h_n'))$ is eventually constant, and that the induced endomorphism $\varphi_n$ of the iterated Stein factor $Z$ is an automorphism for all $n \geq N_{\mathrm{Stein}}$. 
\end{remark}

\begin{remark}\label{algautocor}
Let $(X,f)$ be a dynamical system. The following assertions are equivalent:
        \begin{enumerate}[(1)]
            \item $f$ is a bounded automorphism.
            \item All iterates of $f$ are bounded automorphisms.
            \item Some iterate of $f$ is a bounded automorphism.
        \end{enumerate}
This follows from Theorem \ref{mainthm2} together with the fact that if $f^n$ is an automorphism for some $n$, then $f$ is an automorphism.
\end{remark}

\begin{example}\label{torsionendo}
If $f \in \mathrm{End}_X(k)$ satisfies $f^n \in \mathrm{Aut}^0_X$ for some $n \geq 1$ (for example, $f^n =\mathrm{id}_X$; a torsion endomorphism), then $f$ is a bounded automorphism by Remark \ref{algautocor}.
\end{example}

We conclude with Examples \ref{connsemigroup1} and \ref{connsemigroup2} illustrating the existence of connected semigroup schemes that are not algebraic. This situation does not occur with connected group schemes by \cite[II,\S 5,1.1]{demazure1970groupes}.

\begin{example}\label{connsemigroup1}
Let $P$ be a pencil of infinitely many copies of the multiplicative affine line $\mathbb{A}_k^1$ sharing the same origin. For every $x,y$ in $P$, define $x \star y$ to be the usual product $xy$ if $x$ and $y$ lie on the same copy of $\mathbb{A}_k^1$, and to be zero otherwise. Then $P$ is a connected semigroup scheme. Note that $P$ is not locally algebraic.
\end{example}

\begin{example}\label{connsemigroup2}
Consider an infinite zigzag $Z$ obtained as a union of copies of the multiplicative affine line $\mathbb{A}^1_k$ indexed by $\mathbb{Z}$, such that for every $n \in \mathbb{Z}$, we make the following identifications:
\begin{itemize}
    \item If $n$ is even, then the point $0$ of the $n$th copy is identified with the point $0$ in the next.
    \item If $n$ is odd, then the point $1$ of the $n$th copy is identified with the point $1$ in the next.
\end{itemize}
Define $x \star y$ to be the usual product $xy$ if $x$ and $y$ lie on the same copy of $\mathbb{A}_k^1$, and to be zero otherwise. Then $Z$ is a connected locally algebraic semigroup (clearly not algebraic). See Figure \ref{figureaffinezigzag}.

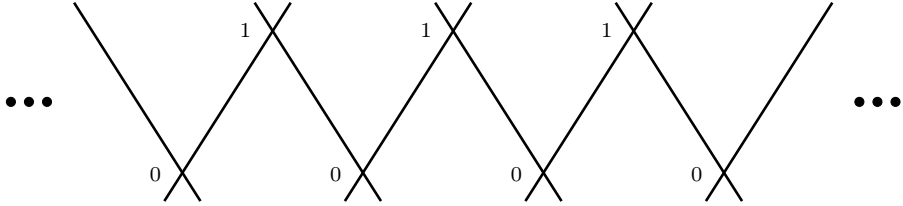
\begin{figure}[ht]
\centering
\begin{tikzpicture}[scale=1]
\def\xmax{8} 
\def\ymax{2} 

\foreach \i in {0,1,2,3} {
    \draw[line width=1pt] (\i*2-0.2, \ymax+0.1) -- (\i*2+1.2, -0.1);
    \draw[line width=1pt] (\i*2+0.8, -0.1) -- (\i*2+2.2, \ymax+0.1);
}

\foreach \i in {0,1,2,3} {
    \node[inner sep=0.5pt] at (\i*2+0.7, 0.2) {\scriptsize 0}; 
    \ifnum\i<3
        \node[inner sep=0.5pt] at (\i*2+1.7, \ymax-0.2) {\scriptsize 1}; 
    \fi
}

\foreach \i in {0,1,2} {
    \filldraw[black] (-0.5 - \i*0.2, 0.5*\ymax) circle (0.05cm);
}

\foreach \i in {0,1,2} {
    \filldraw[black] (\xmax+0.5 + \i*0.2, 0.5*\ymax) circle (0.05cm);
}
\end{tikzpicture}
\caption{An infinite zigzag of affine lines}
\label{figureaffinezigzag}
\end{figure}
\end{example}

\bibliography{sn-bibliography}
\end{document}